\documentclass[10pt, article]{amsart}

\usepackage{tikz}
\usetikzlibrary{calc}
\usepackage{ae} 
\usepackage[T1]{fontenc}
\usepackage[cp1250]{inputenc}
\usepackage{amsmath}
\usepackage{amssymb, amsfonts,amscd,verbatim}

\usepackage[normalem]{ulem}
\usepackage{hyperref}
\usepackage{indentfirst}
\usepackage{latexsym}
\input xy
\xyoption{all}

\usepackage{amsmath}    

\theoremstyle{plain}
\newtheorem{Pocz}{Poczatek}[section]
\newtheorem{Proposition}[Pocz]{Proposition}
\newtheorem{Theorem}[Pocz]{Theorem}
\newtheorem{Corollary}[Pocz]{Corollary}

\newtheorem{Lemma}[Pocz]{Lemma}

\theoremstyle{definition}
\newtheorem{Definition}[Pocz]{Definition}

\theoremstyle{remark}

\def\NN{{\mathbb N}}

\numberwithin{equation}{section}

\title[
Partitions of unity and coverings
]%
  {Partitions of unity and coverings}

\author{Kyle ~Austin}
\address{University of Tennessee, Knoxville, USA}
\email{kaustin9@utk.edu}

\author{Jerzy ~Dydak}
\address{University of Tennessee, Knoxville, USA}
\email{jdydak@utk.edu}

\date{ \today
}
\keywords{Partitions of unity and coverings}

\subjclass[2000]{Primary 54F45; Secondary 55M10}

\begin{document}
\maketitle

\begin{abstract}
The aim of this paper is to prove all well-known metrization theorems using partitions of unity. To accomplish this, we first discuss sufficient and necessary conditions for existence of $\mathcal{U}$-small partitions of unity (partitions of unity subordinated to an open cover $\mathcal{U}$ of a topological space $X$). 
\end{abstract}

\section{Introduction}
The purpose of this paper is to further establish and demonstrate the utility of the calculus of partitions of unity. The setting we chose to this was to prove all the well known metrization theorems in a coherent way which is accessible to students. The coherence stems from the utility of forming partitions of unity on collections of discretely distributed sets. Partitions of unity, viewed as functions to  $l_1(S)$ or simplicial complexes, offer both a geometric and analytical standpoint when working with covers.

It is well known that an open covering $\mathcal{U}$ of a topological space $X$ can be intimately associated with simplicial complexes via mapping $X$ to the nerve of the covering $\mathcal{N(U)}$. The goal of this paper is to display how fundamental this association is. In particular, we stress the relationship between barycentric subdivision with star refinements and covers with $\sigma-$discrete refinements to countable joins of simplicial complexes. It seems to the authors that the correspondence between joins and discreteness of covers is an essential viewpoint when viewing the metrization theorems. 

Traditionally partitions of unity are thought of as a collections of maps $\{f_s:X\to [0,1]: s \in S\}$ so that $\sum\limits_{s\in S}f_s(x)=1$ for every $x\in X$. There is another way to view partitions of unity that proves convenient in many proofs and fundamental in certain cases.

Let $S$ be a nonempty set. $\Delta(S)$ is defined to be the subspace of $l_1(S)$ of functions $f:S\to [0,1]$ such that $\| f(x) \|_{l_1} = 1$ for every $x\in X$. A partition of unity can be defined as a map $f:X\to l_1(S)$ for which the image of $X$ is contained in $\Delta(S)$. This definition is motivated by observation that a partition of unity $\{f_s:X\to [0,1]:s\in S\}$ determines a unique function $f:X \to \Delta(S)$ where $x \mapsto f_x$ and $f_x:S\to [0,1]$ is defined by $f_x(s) = f_s(x)$.

One way of creating partitions of unity on a space $X$ is by normalizing partitions of positive and finite functions $g: X \to (0,\infty)$.

\begin{Definition}\label{NormalizationDef}
Suppose $\{f_s\}_{s\in S}$ is a family of functions $f_s:X\to [0,\infty)$ such that $g=\sum\limits_{s\in S}f_s$ maps $X$ to $(0,\infty)$ and is continuous. The \textbf{normalization} of $\{f_s\}_{s\in S}$ is the partition of unity $\{\frac{f_s}{g}\}_{s\in S}$.
\end{Definition}

\begin{Definition}
Suppose $\mathcal{U}$ is a cover of a topological space $X$. A partition of unity $f=\{f_s\}_{s\in S}$
is called \textbf{$\mathcal{U}$-bounded} (or \textbf{$\mathcal{U}$-small}) if for each $s\in S$
the carrier $f_s^{-1}(0,1]$ of $f_s$ is contained in a subset of $\mathcal{U}$. The carriers $f_s^{-1}(0,1]$ are sometimes reffered to as the \textbf{cozero sets} of the partition.

Given $s \in S$ define the star of $s$ as $st(s) = \{f\in l_1(S):f(s) \neq 0\}$. An alternative definition of a partition of unity $f:X \to l_1(S)$ being \textbf{$\mathcal{U}$-bounded} (or \textbf{$\mathcal{U}$-small}) is that for each $s\in S$
the point inverse of the star $f^{-1}(st(s))$ is contained in a subset of $\mathcal{U}$.

\end{Definition}

For a given space $X$ the existence of $\mathcal{U}$-small partitions of unity for coverings $\mathcal{U}$ is an essential question in geometric topology. In particular, the existence of extensions of partitions of unity defined on closed subspaces is fundamental in general topology, see \cite{DyExtTheory}. 

We are grateful to Szymon Dolecki for asking which metrization theorems can be derived from partitions of unity.

\section{A Unifying Concept}
The purpose of this section is to introduce how we are using partitions of unity to unify discreteness and joins of complexes. A consequence is an alternative formulation and proof of Urysohn's Metrization Theorem. The authors feel that the alternate formulation and proof offer more geometric intuition for metric spaces and, in particular, of the very important metric space $l_1$. Methods in this section forshadow those used throughout the exposition. Consider the classical formulation of Urysohn's Metrization Theorem:

\begin{Theorem}[Urysohn's Metrization Theorem Version 1]
A second countable space is metrizable if and only if it is normal.
\end{Theorem}

A traditional method of proof for this theorem is through embedding into a product of intervals using Tychonoff's Theorem. We offer a different approach using partitions of unity in a way that acts as a prelude to Dydak's Metrization Theorem \ref{DydMetrThm} (see also the book \cite{Dol} of Sz.Dolecki). 

\begin{Theorem}[Urysohn's Metrization Theorem Version 2]\label{Version 2}
A second countable space embeds into $l_1(\NN)$ if and only if $X$ is normal.
\end{Theorem}

\begin{proof}[proof of Version 2]
Notice that $X$ being normal implies there exists a collection of maps $\Delta$ from $X$ into $[0,1]$ so that $\{f^{-1}((0,1]): f \in \Delta\}$ form a basis for $X$. Notice also that $\Delta$ can be chosen to be countable since $X$ is second countable. Denote the elements of $\Delta$ by $f_1,f_2,f_3,\ldots$
Consider the function $f:X \to l_1(\NN)$ by $x \mapsto f_x$ where $f_x(n) = \frac{f_n(x)}{2^n}$. 

To see that $f$ is continuous, observe that $\sum\limits_{n \ge 1}f_n \leq 1$ is finite and for each $\epsilon >0$ there exists a natural $M>0$ so that $\sum\limits_{M < n}f_n(y) < \frac{\epsilon}{4}$  for all $y \in X$. $f_n$ is a continuous function for each $n \leq M$, so there exists a $\delta >0$ so that $|f_n(x)-f_n(y)|< \frac{\epsilon}{2M}$ for all $y\in B(x,\delta)$ and all $n\leq M$. Notice that $|f_x-f_y|_{l_1} = \sum\limits_{n \ge 1} |f_n(x) - f_n(y)| = \sum\limits_{n \ge 1}^{M} |f_n(x) - f_n(y)|+\sum\limits_{M < n} |f_n(x) - f_n(y)| < \sum\limits_{n=1}^M\frac{\epsilon}{2M} + 2\sum_{M < n} f_n(x) \leq \frac{\epsilon}{2} + \frac{\epsilon}{2} = \epsilon$ for each $y \in B(x,\delta)$. 

To see that this map is one-to-one, note that for $x \neq y$ there must exist, by Hausdorff considerations, $n\ge1$ so that $f_n^{-1}((0,1])$ contains $x$ but not $y$. Note that $f_y(n)= 0 \neq f_x(n)$.

This map is an embedding as it is continuous and the preimage of a basis refines a basis by assumption.
\end{proof}

Notice that $X$ can inherit a metric when considered as a subspace of $l_1(S)$ defined as $d(x,y) = ||f_x-f_y||= \sum\limits_{s\in S}|f_x(s)-f_y(s)|$. The second author noticed that this same technique could be used to define metrics on spaces given a partition of unity whose carriers form a basis. Urysohn's Metrization Theorem could be proven in a class setting and stand as motivation to prove the following metrization theorem which was introduced in \cite{DyPU}. 

\begin{Theorem}[Dydak's Metrization Theorem \cite{DyPU}]\label{DydMetrThm}
A space $X$ is metrizable if and only if for there exists a continuous partition of unity $f:X\to l_1(V)$ so that $\{f^{-1}(st(v)): v \in V\}$ forms a basis for $X$. 
\end{Theorem}

The proof of Version 2 of Urysohn's Metrization Theorem can be modified to give a different classification of metric spaces.

\begin{Corollary}
A space $X$ is metrizable if and only if there exists a set $V$ and an embedding of $X$ into $l_1(V)$, so that each element in the image has norm 1.
\end{Corollary}
\begin{proof}
Suppose $X$ is metrizable. By Dydak's Metrization Theorem \ref{DydMetrThm}, there exists a set V and a continuous partition of unity $f:X\to l_1(V)$ so that $\{f^{-1}(st(v)): v \in V\}$ forms a basis for $X$. This $f$ must be one-to-one by Hausdorff considerations. It is an embedding as the inverse image of a basis for $l_1(V)$ refines a basis for $X$.
\end{proof}

\section{Star refinements}
This section is devoted to creation of partitions of unity using star refinements. Our major application is the Birkhoff Metrization Theorem for Groups \ref{BirkhoffMetrThm}.
\begin{Definition}
A cover $\mathcal{V} $ of a set $X$ is a \textbf{star-refinement} of a cover $\mathcal{U} $ if for each $x\in X$ the \textbf{star} $st(x,\mathcal{V} )=\bigcup \{V\in \mathcal{V}  | x\in V\}$ is contained in an element of $\mathcal{U} $.
\end{Definition}

Notice that for a simplicial complex $K$ with vertex set $S$, the cover by open stars of vertices of $S$ has a natural star refinement given by starring of the vertices of the barycentric subdivision $K'$ of $K$. In \cite{DyPU}, the second author shows how to use barycentric subdivisions of simplicial complexes to find star refinements of covers with the concept of derivatives of partitions of unity.

\begin{Definition}
Let $X$ be a space, $S \neq \emptyset$, and $\{f_s\}_{s\in S}$ be a partition of unity on $X$. The $\textbf{derivative}$ of $\{f_s\}_{s\in S}$ is a partition of unity $\{f'_T\}_{T \subset S}$ indexed by all finite subsets of $S$ so that

1) $\sum_{s \in T}\frac{f'_T}{|T|} = f_s$ for each $s \in S$ 

2) $f'_T(x) \neq 0$ and $f'_F(x) \neq 0$ implies $T \subset F$ or $F \subset T$.
\end{Definition} 

The definition above is modeled after the special case for a continuous partition of unity $f:X \to |K|$ where the derivative is the induced map to the barycentric subdivision of $K$. See \cite{DyPU} for a constructive proof that derivatives of arbitrary partitions of unity always exist and are unique. The following proposition, proved in \cite{DyPU}, shows how to obtain star refinements from partitions of unity.

\begin{Proposition}[\cite{DyPU}]\label{ExistenceOfStarRef}
Suppose  $f: X \to l_1(S)$ is a partition of unity and $f':X \to l_1(\{T \subset S:|T|<\infty\})$ is its derivative. The open cover $\{(f')^{-1}(st(T)) | T \subset S:|T|<\infty\}$ of $X$ is a star refinement of the cover $\{f^{-1}(st(s)) | s\in S\}$  of $X$.
\end{Proposition}

The following Theorem exhibits that star refinements of covers can be used to construct partitions of unity.

\begin{Theorem}\label{StarTheorem}
Suppose $\mathcal{U} $ is an open cover of a topological space $X$. $\mathcal{U}$-small partitions of unity exists if and only if
there is a sequence of open covers $\mathcal{U}_n$ of $X$ such that $\mathcal{U}_1=\mathcal{U}$ and $\mathcal{U}_{n+1}$ is a star refinement of $\mathcal{U}_n$ for each $n$.
\end{Theorem}
\begin{proof}
The converse follows by Proposition \ref{ExistenceOfStarRef} and iteration.

Suppose there is a sequence of open covers $\mathcal{U}_n$ of $X$ such that $\mathcal{U}_1=\mathcal{U}$ and $\mathcal{U}_{n+1}$ is a star refinement of $\mathcal{U}_n$ for each $n$. By picking only odd-numbered covers we may assume that for each $U\in \mathcal{U}_{n+1} $ the \textbf{star} $st(U,\mathcal{U}_{n+1} )=\bigcup \{V\in \mathcal{U}_{n+1}   | U\cap V\ne\emptyset\}$ is contained in an element of $\mathcal{U}_n $

Our strategy is to find a metric $d$ on a quotient space $X/{\sim}$ of $X$ such that the projection $p:X\to X/{\sim}$ is continuous and there is an open cover $\mathcal{V} $ of $X/{\sim}$ with the property that $p^{-1}(\mathcal{V} )$ refines $\mathcal{U} $. As there is a $\mathcal{V} $-small partition of unity $g:X/{\sim}\to l_1(S)$, $g\circ p$ is a $\mathcal{U} $-small partition of unity on $X$.
\par The relation $x\sim y$ is defined by the property that for each $n$ there is $U\in \mathcal{U}_n$ containing both $x$ and $y$. It is an equivalence relation. Our first approximation of a metric on $X/{\sim}$ is $\rho(x,y)$ defined
as infimum of $2^{-n}$ such that there is $U\in \mathcal{U}_n$ containing both $x$ and $y$. 
$d(x,y)$ is defined as the infimum of all sums $\sum\limits_{i=1}^k \rho(x_i,x_{i+1})$, where $x_1=x$ and $x_{k+1}=y$.

\textbf{Claim 1}: If $\sum\limits_{i=1}^k \rho(x_i,x_{i+1}) < 2^{-n}$ for some chain of points where $x_1=x$, $x_{k+1}=y$, and $n\ge 1$,
then $\rho(x,y)\leq 2^{-n}$.
\par \textbf{Proof of Claim 1}: We may assume that $k \ge 1$ is the smallest natural number such that there is a chain of size $k-1$ joining $x$ and $y$ for which the sum $\sum\limits_{i=1}^k \rho(x_i,x_{i+1})$ is smaller than $2^{-n}$. Notice the lengths $\rho(x_i,x_{i+1})$ and $\rho(x_{i+1},x_{i+2})$ of adjacent links must be different. Otherwise if there was some $m$ such that $\rho(x_i,x_{i+1}) = \rho(x_{i+1},x_{i+2}) = 2^{-m}$ then there would be elements of $\mathcal{U}_m$ containing the pairs $\{x_i,x_{i+1}\}$ and $\{x_{i+1},x_{i+2}\}$. $\mathcal{U}_{m}$ star refines $\mathcal{U}_{m-1}$ which means $\rho(x_i,x_{i+2}) \leq 2^{-m+1} = \rho(x_i,x_{i+1}) + \rho(x_{i+1},x_{i+2})$ and so we could drop  the middle term $x_{i+1}$ without increasing the sum.

Notice also that we cannot have the situation of three links such that sizes of the first and the third are equal and bigger than the size of the middle one. To see this, let $m\ge 1$ be the largest natural number so that the first pair and the last pair are contained in some element of $\mathcal{U}_m$.Then each of the three pairs of points are contained in some element of $\mathcal{U}_m$. Let $U$ be the element containing the middle pair. $\mathcal{U}_{m-1}$ star refines $\mathcal{U}_m$ and so $st(U,\mathcal{U}_{m})$ is contained in an element of $\mathcal{U}_{m-1}$ which means the distance between the first and last point is at most $2^{-m+1}$ and so is strictly less than the sum of the three distances.
\newline
\indent
Claim is clear for $k\leq 3$. We proceed by induction. Pick $m < k$ such that $\sum\limits_{i=1}^m \rho(x_i,x_{i+1}) < 2^{-n-1}$ but $\sum\limits_{i=1}^{m+1} \rho(x_i,x_{i+1}) \ge 2^{-n-1}$. Notice 
$\sum\limits_{i=m+1}^k \rho(x_i,x_{i+1}) < 2^{-n-1}$. By induction, $\rho(x,x_m)\leq 2^{-n-1}$ and $\rho(x_{m+1},y)\leq 2^{-n-1}$. Since there is $U\in \mathcal{U}_{n+1}$ containing $x_m$ and $x_{m+1}$, the star $st(U,\mathcal{U}_{n+1} )$ contains both $x$ and $y$.

\textbf{Claim 2}: $p:X\to X/{}\sim$ is continuous.
\par \textbf{Proof of Claim 2}: If $d(x,y) < r$ pick $n\ge 1$ such that $r-d(x,y) > 2^{-n}$.
Choose $U\in \mathcal{U}_n$ containing $y$ and notice $d(x,z) < r$ for any $z\in U$. That means $p^{-1}(B(x,r))$ is open in $X$ and $p$ is continuous.

\textbf{Claim 3}: $p^{-1}(B(x,1/2))\}_{x\in X}$ refines $\mathcal{U}$.
\par \textbf{Proof of Claim 3}: Choose $V\in\mathcal{U}_2$ containing $x\in X$. If $d(x,y) < 1/2$, then Claim 1
says $\rho(x,y)\leq 1/2$ and there is $W\in \mathcal{U}_2$ containing $x$ and $y$.
Therefore $B(x,1/2)\subset st(V,\mathcal{U}_2)$ which is contained in an element of $\mathcal{U}$.
\end{proof}

\begin{Corollary}
If $U$ is a non-empty open subset of a topological group $G$, then the cover $\mathcal{U}=\{g\cdot U\}_{g\in G}$
has a $\mathcal{U}$-small partition of unity.
\end{Corollary}
\begin{proof}
We may assume $U$ contains $1_G\in G$. Pick a sequence $\{U_n\}_{n\ge 1}$ of symmetric open neighborhoods of $1_G\in G$ such that $U_{n+1}\cdot U_{n+1}\subset U_n\subset U$ for all $n\ge 1$.

Consider the coverings $\mathcal{U}_n=\{g\cdot U_n\}_{g\in G}$. $ \mathcal{U}_{n+1}$ is a star refinement of $ \mathcal{U}_{n-1}$. To see this, it suffices to show that $st(g \cdot U_{n+1},\mathcal{U}_n) \subset g \cdot U_{n-1}$ for each $g \in G$. Let $g$ and $h$ be in  $G$ with $g \cdot U_{n+1} \cap h \cdot U_{n+1} \neq \emptyset$. Then there are $u$ and $v$ in $U_{n+1}$ so that $g\cdot u = h\cdot v$. This implies $g^{-1}h \in U_{n+1}\cdot U_{n+1} \subset U_n$ and so $g^{-1}\cdot h\cdot U_{n+1} \subset g^{-1}\cdot h\cdot U_{n} \subset U_{n-1}$. It follows that $h\cdot U_{n+1} \subset g\cdot U_{n-1}$ which proves  $st(g \cdot U_{n+1},\mathcal{U}_n) \subset g \cdot U_{n-1}$.  So using every other cover we have, by \ref{StarTheorem}, a $ \mathcal{U}$-small partition of unity $f$ on $G$. 

\end{proof}

\begin{Corollary}[Birkhoff Metrization Theorem for Groups \cite{Nag}: page 225] \label{BirkhoffMetrThm}
A topological group $G$ is metrizable if and only if it is first countable.
\end{Corollary}
\begin{proof}
Every metrizable space is first countable, so assume $G$ is first countable and pick a basis $\{U_n\}_{n\ge 1}$ of  open neighborhoods of $1_G\in G$ .
Let $\mathcal{U}_n=\{g\cdot U_n\}_{g\in G}$. There are $ \mathcal{U}_n$-small partitions of unity $\{f_{n,g}\}_{g \in G}$ on $G$. This leads to a partition of unity $\{\frac{f_{n,g}}{2^n}:g \in G$ and  $n \in \mathbb{N}\}$ on $G$ whose carriers form a basis of $G$. Dydak's Metrization Theorem \ref{DydMetrThm} says $G$ is metrizable if it is Hausdorff.
\end{proof}

\section{Sigma Refinements}

In this section we show how to construct partitions of unity using $\sigma$-discrete refinements of open covers.

For a normal space $X$ and a finite open cover $\mathcal{U}$, a $\mathcal{U}$-small partition of unity can be obtained as follows: Choose a closed subset $B_U$ of each $U \in \mathcal{U}$ and use Urysohn's Lemma to get a map $f_U : X\to [0,1]$ where $f_U(B_U) \subset \{1\}$ and $f_U(X \setminus U) \subset \{0\}$. $\sum\limits_{U \in \mathcal{U}}f_U$ is a partition of a continuous function $g: X \to (0,\infty)$. Notice that $\{\frac{f_U}{g}\}_{U \in \mathcal{U}}$ is a $\mathcal{U}$-small partition of unity on $X$.

Suppose now that $\mathcal{U}$ is a countable open cover of $X$. $\mathcal{U}$ may be written as a $\bigcup\limits_{n \ge1} \mathcal{U}_n $ where $\mathcal{U}_n$ is a finite subset of $\mathcal{U}$ for $n \ge 1$. For each $n \ge 1$ we can find a collection of functions $\{f_U\}_{U \in \mathcal{U}_n}$ using Urysohn's Lemma as done above. Notice that $\{\frac{f_U}{2^n}: U \in \mathcal{U}_n$ and $n \ge1 \}$ is a partition of a positive and finite function $h:X\to (0,\infty)$. Normalizing this family yields a $\mathcal{U}$-small partition of unity on $X$.

The previous two methods can be summarized as follows: Given a cover $\mathcal{U}$ of a space $X$, find a countably many partitions $\{f_s\}_{s\in S_n}$ of positive and finite functions on $X$ so that $\{f_s^{-1}(0,1]: s \in S_n\text{ and }n \ge1\}$ refines $\mathcal{U}$. These collections induce the partition $\{\frac{f_s}{2^n}:$$s \in S_n$ and $n\ge 1 \}$ of a positive and finite function. Normalizing this collection yields the desired $\mathcal{U}$-small partition of unity.

Recall that a space $X$ is $\textbf{collectionwise normal}$ if every discrete collection of closed subsets $\{A_s\}_{s\in S}$ has a discrete open coarsening $\{U_s\}_{s\in S}$ where $A_s \subset U_s$ for all $s\in S$.

\begin{Proposition}\label{SigmaDiscreteOpenRefinementProp}
A $\sigma$-discrete open cover $\mathcal{U} = \{U_s\}_{s\in S}$ of a normal space $X$ admits a $\mathcal{U}$-small partition of unity if and only if it has a $\sigma$-discrete closed refinement. 
\end{Proposition}
\begin{proof} $(\Leftarrow)$ 
$\mathcal{U} = \bigcup\limits_{n \ge 1} \mathcal{U}_n$ where  $\mathcal{U}_n = \{U_{n,s} \subset U_s:s\in S\}$ is a discrete open collection. By intersecting $\mathcal{U}_n$ with families in the $\sigma-$discrete closed refinement and taking closures, we may assume that $\mathcal{U}_n$ has a $\sigma-$discrete closed refinement for each $n\ge 1$. Label the $\sigma-$discrete closed refinement of $\mathcal{U}_n$ by $\mathcal{D}_n = \bigcup\limits_{m\ge 1}\{D_{s,m} \subset U_{s,n}:s\in S\}$. For each $\mathcal{U}_n$ we can obtain a partition of a positive and finite function whose cozero sets refines it as follows. For each $s\in S$ and $m \ge 1$ we can use Urysohn's Lemma to get a function $f_{s,m}:X \to [0,1]$ where $D_{s,m}$ is mapped to $1$ and $X \setminus U_{n,s}$ is mapped to $0$. For each $m$, we therefore have a collection of functions $\{f_{m,s}:s\in S\}$. Consider the family $\{ g_{n,s} = \sum\limits_{m\ge 1}\frac{f_{m,s}}{2^m}:s \in S\}$. This is a partition of a positive and finite function whose cozero sets are precisely $\mathcal{U}_n$. 

Consider the family $\{h_s = \sum\limits_{n\ge 1} \frac{g_{n,s}}{2^n}:s\in S\}$. This is a $\mathcal{U}-$small partition of a positive and finite function. Normalizing this family yields a $\mathcal{U}-$small partition of unity. 

$(\Rightarrow)$ See Proposition \ref{ExistenceSigmaDiscreteRefinementProp} for the converse.
\end{proof}

\begin{Corollary}\label{SigmaDiscreteRefinementProp}
Assume $X$ is collectionwise normal. An open cover $\mathcal{U} = \{U_s\}_{s\in S}$ of $X$ admits a $\mathcal{U}$-small partition of unity if and only if it has a $\sigma$-discrete closed refinement. 
\end{Corollary}
\begin{proof}
$(\Leftarrow)$ Choose a $\sigma$-discrete closed refinement $\mathcal{V} = \bigcup\limits_{n=1}^{\infty}\mathcal{V}_n$ of $\mathcal{U}$ where each $\mathcal{V}_n = \{F_{s}\}_{s \in S_n}$ is a discrete collection of closed subsets of $X$. Let $\{U_{s}\}_{s\in S_n}$ be a discrete open collection with $F_{s} \subset U_{s}$ and $U_{s}$ lies in some element of $\mathcal{U}$ for each $s \in S_n$. Using Urysohn's Lemma for each $s \in S_n$ there is a function $f_{s}:X \to [0,1]$ where $f_{s} (X \setminus U_{s}) \subset \{0\}$ and $f_{s}(F_s) \subset \{1\}$. The collection of functions $\bigcup\limits_{n=1}^{\infty}\{\frac{f_{s}}{2^n}:s \in S_n\}$ form a partition of a positive and finite funtion $g:X \to (0,\infty)$. Normalizing the family yields a partition of unity on $X$ that is $\mathcal{U}$-small.

$(\Rightarrow)$ See \ref{ExistenceSigmaDiscreteRefinementProp} for the converse.
\end{proof}

\begin{Proposition} \label{Collnorm-sigmarefinementsTheorem}
Assume $X$ is collectionwise normal. An open cover $\mathcal{U}$ of $X$ admits a $\mathcal{U}$-small partition of unity if and only if there is an open refinement of $\mathcal{U}$ that is point-finite. 
\end{Proposition}
\begin{proof}
Choose a point-finite open refinement $\mathcal{V}=\{V_s\}_{s\in S}$ of $\mathcal{U}$. Let $S(n)$ be the family of subsets of $S$
consisting of eactly $n$ elements.

For each $n\ge 1$ we will construct
two discrete closed shrinkings $\mathcal{D}_n=\{D_{T}\}_{T\in S(n)}$ and $\mathcal{D}^\ast_n=\{D^\ast_{T}\}_{T\in S(n)}$ of $\mathcal{V}$ such that $\bigcup \mathcal{D}^\ast_n$
is a cover of $X$.

$D_{s}$ is the set of points that belong to $V_s$ only: $D_{s}:=V_s\setminus \bigcup\limits_{t\ne s}V_t$.
It is a closed subset of $V_s$ and the only element of $\mathcal{V}$ that may intersect it is $V_s$, so $\{D_{s}\}_{s\in S}$ is a discrete closed family. We extend $\{D_{s}\}_{s\in S}$ to a discrete family $\{D^\ast_{s}\}_{s\in S}$
such that $D_s\subset int(D^\ast_s)\subset D_s^\ast \subset V_s$ for all $s\in S$.

Suppose $\mathcal{D}_n$ and  $\mathcal{D}^\ast_n$ are known for all $n < k$.
Given $T\in S(k)$ define $D_{T,k}$ as the set of points belonging to $\bigcup\limits_{t\in T} V_t$
and not belonging to $\bigcup\limits_{t\notin T}V_t\cup \bigcup\limits_{n < k}\bigcup\limits_{F\in S(n)}int(D^\ast_F)$. Let $\mathcal{D}_k$ be the collection of $D_{T,k}$ where $T$ ranges over $S(k)$.

It is a discrete closed collection that we can enlarge to $D^\ast_k$. Notice that all the points contained in at most $n$ elements of the cover are contained in $\bigcup_{k \leq n}D^\ast_k$. Therefore $\bigcup \mathcal{D}^\ast_n$ is indeed a cover of $X$ as $\mathcal{V}$ is point finite. We are done by Proposition \ref{SigmaDiscreteRefinementProp}.
\end{proof}

Recall the notion of weak paracompactness.
\begin{Definition}
A topological space $X$ is \textbf{weakly paracompact} if for every open cover $\mathcal{U} = \{U_s:s\in S\}$ of $X$ there exists a point finite partition of unity $f:X\to l_1(S)$ on $X$. This means that for each $x \in X$ the support of $f(x)$ is finite, i.e. the cardinality of $\{s\in S:f(x)(s) \neq 0\}$ is finite.
\end{Definition}

\begin{Corollary} [Michael-Nagami Theorem \cite{En}]
Assume $X$ is collectionwise normal. $X$ is paracompact if and only if it is weakly paracompact.
\end{Corollary}

\section{Discretization of covers}

The purpose of this section is to broaden the scope of the previous section and illuminate how partitions of unity can be used to unify all metrization theorems. 

In this section, star refinements are used to descritize open covers $\mathcal{U}$ into disjoint closed collections and disjoint open coarsenings of those collections. $\mathcal{U}$-small partitions of unity can be constructed analogously to those in section 3.

Our first construction is that of a discrete shrinking of a given well-ordered open cover $\mathcal{V}=\{V_s\}_{s\in S}$
to a discrete family $\{D_s\}_{s\in S}$ of closed sets so that each element of a given open cover $\mathcal{U}$
intersects at most one element of $\{D_s\}_{s\in S}$. 

\begin{Definition}
Suppose $\mathcal{U}$ is an open cover of a topological space $X$ and $\mathcal{V}=\{V_s\}_{s\in S}$ is a well-ordered collection of open subsets of $X$ (i.e. the index set $S$ is well-ordered). Let $\bar S=S\cup \{\infty\}$ is a well ordered set
obtained from $S$ by adding an element $\infty$ that is larger than all elements of $S$.
Define $\alpha:\mathcal{U}\to \bar S$ as follows: $\alpha(U)=\infty$ if there is no $s\in S$ with $U \subset V_s$.
Otherwise $\alpha(U)$ is the smallest $s\in S$ satisfying $U \subset V_s$.

The 
\textbf{discretization} $D(\mathcal{V},\mathcal{U})=\{D(\mathcal{V},\mathcal{U})_s\}_{s\in S}$ of $\mathcal{V}$
 with respect to $\mathcal{U}$ is defined as follows: $D(\mathcal{V},\mathcal{U})_s=X\setminus \bigcup\limits_{\alpha(U)\ne s}\{U: U \in \mathcal{U}\}$.
\end{Definition}

The following result confirms that the
discretization $D(\mathcal{V},\mathcal{U})$ is indeed a discrete closed shrinking of $\mathcal{V}$.

\begin{Lemma}\label{CNDiscClosedLemma}
$D(\mathcal{V},\mathcal{U})_s\subset V_s$ and $U\cap D(\mathcal{V},\mathcal{U})_s\ne\emptyset$ implies $\alpha(U)=s$ for each $s\in S$.
\end{Lemma}
\begin{proof} If $\alpha(U)\ne s$, then $U\subset X\setminus D(\mathcal{V},\mathcal{U})_s$, so $U\cap D(\mathcal{V},\mathcal{U})_s\ne\emptyset$ implies $\alpha(U)=s$. Also, any $W\in\mathcal{U}$ containing a point outside of $V_s$ has $\alpha(W)\ne s$, resulting in
$W\subset X\setminus D(\mathcal{V},\mathcal{U})_s$. Thus, $D(\mathcal{V},\mathcal{U})_s\subset V_s$.
\end{proof}

\begin{Lemma}\label{CNDiscToRefinementLemma}
Suppose $\mathcal{V}=\{V_s\}_{s\in S}$ is a cover of $X$ and $ \{\mathcal{U}_n\}_{n\ge 1}$ is a sequence of open covers of $X$.  If for any $x\in V_s\in \mathcal{V}$ there is $n\ge 1$ satisfying $st(x,\mathcal{U}_n)\subset V_s$, then the union of discretizations 
$\bigcup D(\mathcal{V},\mathcal{U}_n)$ is a cover of $X$.
\end{Lemma}

\begin{proof}
Given $x\in X$, let $s\in S$ be the smallest element so that $x\in V_s$. There is $n$ such that $st(x,\mathcal{U}_n)\subset V_s$. Now, $x$ belongs to the element of $D(\mathcal{V},\mathcal{U}_n)$ indexed by $s$.
Indeed, $\alpha(U)=s$ for any $U\in \mathcal{U}_n$ containing $x$.
\end{proof}

\begin{Proposition}\label{ExistenceSigmaDiscreteRefinementProp}
If an open cover $\mathcal{U}$ of a topological space $X$ has a $\mathcal{U}$-small partition of unity, then $\mathcal{U}$ admits a $\sigma$-discrete closed refinement.
\end{Proposition}
\begin{proof}
Choose a $\mathcal{U}$-small partition of unity $f:X \to l_1(S)$. Set $\mathcal{V} =\{ st(s): s \in S \}$ and notice that $\mathcal{V}$ is a cover of the metrizable space $l_1(V) \setminus \{0\}$. By using Theorem \ref{StarTheorem}, $l_1(V) \setminus \{0\}$ has a sequence of open covers $\mathcal{U}_n$ such that for any $x \in U \in \mathcal{V}$ there is $n \ge 1$ satisfying $st(x,\mathcal{U}_n)\subset U$. Pulling these covers back to $X$, we get a sequence of open covers $\mathcal{V}_n$ such that for any $y \in V \in \mathcal{U}$ there is $n \ge 1$ so that $st(x,\mathcal{V}_n) \subset V$. Using lemma \ref{CNDiscToRefinementLemma} and Lemma \ref{CNDiscClosedLemma} we can use discretizations to get a $\sigma$-discrete closed refinement of $\mathcal{U}$.
\end{proof}

\begin{Theorem}\label{SubmainDiscreteTheorem}
Suppose $X$ is collectionwise normal and $\mathcal{U}$ is an open cover of $X$. $\mathcal{U}$-small partitions of unity exists if there is a sequence of covers $\mathcal{U}_n$ such that for any $x\in U\in \mathcal{U}$ there is $n\ge 1$ satisfying $st(x,\mathcal{U}_n)\subset U$. 
\end{Theorem}
\begin{proof}
($\Rightarrow$ ) Well order the cover $\mathcal{U} = \{U_s\}_{s\in S}$. Using Lemmas \ref{CNDiscClosedLemma} and \ref{CNDiscToRefinementLemma}, $\mathcal{U}$ has a $\sigma$-discrete closed refinement $\bigcup\limits_{n\in \NN} D(\mathcal{U},\mathcal{U}_n)$. Using collectionwise normality, $D(\mathcal{U},\mathcal{U}_n)$ has a discrete open coarsening of $\mathcal{V}_n $. Apply \ref{SigmaDiscreteOpenRefinementProp}.

($\Leftarrow$) This direction follows from Proposition \ref{StarTheorem}.
\end{proof}

\begin{Corollary}[Bing Criterion \cite{En}] \label{BingCriterion}
Suppose $\{\mathcal{U}_n\}_{n\ge 1}$ is a sequence of open covers of $X$ such that $\{st(x,\mathcal{U}_n)\}$ is a basis at $x$ for each $x\in X$. If $X$ is collectionwise normal, then it is metrizable.
\end{Corollary}

\begin{proof}
This follows from Theorem \ref{SubmainDiscreteTheorem} and Dydak's Metrization Theorem \ref{DydMetrThm}. Indeed, in this case any open cover of $X$ admits a partition of unity. In particular, each $\mathcal{U}_n$ admits a partition of unity $f_n$. Those, when combined into one partition of unity $f$ result in $f$ having point-inverses of stars of vertices being a basis of $X$. By Dydak's Metrization Theorem \ref{DydMetrThm}, $X$ is metrizable.
\end{proof}

Given an open cover $\mathcal{U}$ of $X$ let $st(\mathcal{U})$ be the cover $\{st(x,\mathcal{U})\}_{x\in X}$ of $X$ by stars of $\mathcal{U}$.

\begin{Theorem}\label{MainDiscretTheorem}
Suppose $X$ is normal and $\mathcal{U}$ is an open cover of $X$. $\mathcal{U}$-small partitions of unity exists if and only if there is a sequence of open covers $\mathcal{U}_n$ such that for any $x\in X$ and any open neighborhood $U\in \mathcal{U}$ of $x$ there is $n\ge 1$ satisfying $st(x,st(\mathcal{U}_n))\subset U$.
\end{Theorem}
\begin{proof}

The main idea is to discretize the cover $\mathcal{U}$ into countable many discrete open collections. Well order the cover $\mathcal{U} = \{U_s:s\in S\}$. 
Let $\mathcal{V}_n=\{st(U,\mathcal{U}_n)\}_{U\in \mathcal{U}_n}$.
Consider discretizations $D(\mathcal{U},\mathcal{V}_n)$. First of all we need $\bigcup\limits_{n=1}^\infty D(\mathcal{U},\mathcal{V}_n)$ to be a cover of $X$.
Indeed, given $x\in X$ choose the smallest $s\in S$ such that $x\in U_s$. Pick $n\ge 1$ satisfying $st(x,st(\mathcal{U}_n))\subset U_s$.
That implies $st(V,st(\mathcal{U}_n))\subset U_s$ for every $V\in \mathcal{U}_n$ containing $x$ and $s$ is the smallest element of $S$ so that
$st(V,st(\mathcal{U}_n))$ is contained in $U_s$. Thus $x\in D(\mathcal{U},\mathcal{V}_n)_s$.

To complete the proof using \ref{SigmaDiscreteOpenRefinementProp} it suffices to show that the stars of elements of $D(\mathcal{U},\mathcal{V}_n)$ with respect
to $\mathcal{U}_n$ form a discrete family. Indeed, suppose $U\in \mathcal{U}_n$ intersects two different sets $st(D(\mathcal{U},\mathcal{V}_n)_s,\mathcal{U}_n)$
and $st(D(\mathcal{U},\mathcal{V}_n)_t,\mathcal{U}_n)$. That means $st(U,\mathcal{U}_n)$ intersects both $D(\mathcal{U},\mathcal{V}_n)_s$
and $D(\mathcal{U},\mathcal{V}_n)_t$. That is not possible.

$(\Leftarrow)$ This direction follows from \ref{StarTheorem}.

\end{proof}

The following is a version of Bing Criterion \ref{BingCriterion} for $T_0$ spaces.
\begin{Theorem}\label{AStarStarMetrTheorem}
A $T_0$ space $X$ is metrizable if and only if there is a sequence of open covers $\mathcal{U}_n$ such that for any $x\in X$
the sets $\{st(x,st(\mathcal{U}_n))\}_{n\ge 1}$ form a basis at $x$. 
\end{Theorem}

\begin{proof}
If $X$ is metrizable, then $\mathcal{U}_n$ consisting of open balls of radius $\frac{1}{n}$ gives the desired sequence of covers.

As long as $X$ is normal, by \ref{MainDiscretTheorem} $X$ is metrizable. Indeed, in this case any open cover of $X$ admits a partition of unity. In particular, each $\mathcal{U}_n$ admits a partition of unity $f_n$. Those, when combined into one partition of unity $f$ result in $f$ having point-inverses of stars of vertices being a basis of $X$. By Dydak's Metrization Theorem \ref{DydMetrThm}, $X$ is metrizable.

We may assume $\mathcal{U}_{n+1}$ is a refinement of $\mathcal{U}_n$ by taking intersections of the first $n$ covers.

Let us show that for any closed set $A$ of $X$ and any $x\notin A$ there is $n$ such that $st(x,\mathcal{U}_n)$ is disjoint with $st(A,\mathcal{U}_n)$.
That is accomplished by choosing $n$ so that $st(x,st(\mathcal{U}_n))\subset X\setminus A$.

Notice $X$ is Hausdorff. Indeed, for any two points $x\ne y$ in $X$ there is a neighborhood $U$ containing only one of them, say $x$.
Apply the above observation to $A=X\setminus U$. Hence $X$ is regular as well.

To prove normality of $X$ assume $A$ and $B$ are disjoint closed sets. For each $x\in X$ choose $n(x)\ge 1$ so that 
$st(x,\mathcal{U}_{n(x)})$ is contained in either $X\setminus st(A,\mathcal{U}_{n(x)})$ or in $X\setminus st(B,\mathcal{U}_{n(x)})$.  Let $V_x$ be an element of $\mathcal{U}_{n(x)}$ containing $x$. Put $\mathcal{V}=\{V_x\}_{x\in X}$. We claim $st(A,\mathcal{V})\cap st(B,\mathcal{V})=\emptyset$. Assume $z\in st(A,\mathcal{V})\cap st(B,\mathcal{V})$. Therefore we have $x\in A$ and $y\in B$ so that
$z\in V_y\cap V_x$. Without loss of generality assume $n(x)\ge n(y)$. In this case $z\in st(y,\mathcal{U}_{n(x)})\cap st(A,\mathcal{U}_{n(x)})$,
a contradiction.
 \end{proof}

\begin{Corollary}[Moore Metrization Theorem \cite{En}]
A $T_0$ space $X$ is metrizable if and only if there is a sequence of open covers $\mathcal{U}_n$ such that for any $x\in X$ and any open neighborhood $U$ of $x$ there is $n\ge 1$ and a neighborhood $V$ of $x$ satisfying $st(V,\mathcal{U}_n)\subset U$. 
\end{Corollary}

\begin{proof}
If $X$ is metrizable, then $\mathcal{U}_n$ consisting of open balls of radius $\frac{1}{n}$ gives the desired sequence of covers.

Given such a sequence of covers one readily sees that for any $x\in X$
the sets $\{st(x,st(\mathcal{U}_n))\}_{n\ge 1}$ form a basis at $x$. Apply \ref{AStarStarMetrTheorem}.
 \end{proof}

\begin{Corollary}
[Arkhangelskii Metrization Theorem \cite{En}]
A space $X$ is metrizable if and only if it is $T_1$ and has a basis $\mathcal{B}$ with the property that 
for any $x\in X$ and any open neighborhood $U$ of $x$ there is a neighborhood $V$ of $x$ in $U$
so that only finitely many elements of $\mathcal{B}$ intersects both $V$ and $X\setminus U$.

\begin{proof}
($\Rightarrow$) Let $\mathcal{V}_n$ be a locally finite refinement of the cover of $X$ by $2^{-n}$ balls. Observe that $\mathcal{B} =\bigcup\limits_{n=1}^{\infty} \mathcal{V}_n$ is a basis for $X$. Let $x \in X$ and $U$ an open neighborhood of $x$. Let $n \ge 1$ be large enough so that for all $ m \ge 1$ no element of $\mathcal{V}_m$ intersects both the ball of radius $2^{-m}$ about $x$ and $X \setminus U$.  By the local finiteness of the covers $\mathcal{V}_n$ we can find an open set $V \subset B(x,2^{-n})$ so that only finitely many elements of $\bigcup\limits_{n=1}^{m}\mathcal{V}_n$ intersect $V$. 

($\Leftarrow$) Choose a basis $\mathcal{B}$ with the property that for any $x\in X$ and any open neighborhood $U$ of $x$ there is a neighborhood $V$ of $x$ in $U$ so that only finitely many elements of $\mathcal{B}$ intersects both $V$ and $X\setminus U$. We will build a sequence of covers $\mathcal{U}_n$ in order to apply Moore's Metrization Theorem.

Let the cover $U_1$ be all the maximal elements of $\mathcal{B}$. Inductively we can build a cover $\mathcal{U}_n$ by first removing all the elements of $\mathcal{U}_k$ for $k < n$ from $\mathcal{B}$ and setting $\mathcal{U}_n$ to be all the maximal elements in the remaining collection.

Notice that for any $x \in X$ and each neighborhood $U$ of $x$ there is an element $V \in \mathcal{B}$ so that $x \in V \subset U$ and only finitely many elements of $\mathcal{B}$ intersect $V$  and $X\setminus U$. Choose $n \ge 1$ large enough so that no element of $\mathcal{U}_n$ intersects both $V$ and $X \setminus U$ and notice that $st(V,\mathcal{U}_n) \subset U$.

\end{proof}

\end{Corollary}

\begin{Corollary}
[Alexandroff Criterion \cite{En}]
A space $X$ is metrizable if and only if it is collectionwise normal and has a basis $\mathcal{B}$ with the property that 
for any $x\in X$ and any open neighborhood $U$ of $x$ only finitely many elements of $\mathcal{B}$ intersects $X\setminus U$ and contain $x$.

\end{Corollary}

\begin{proof}

($\Leftarrow$)  Choose a basis $\mathcal{B}$ with the property that for any $x\in X$ and any open neighborhood $U$ of $x$ there is a neighborhood $V$ of $x$ in $U$ so that only finitely many elements of $\mathcal{B}$ intersects both $V$ and $X\setminus U$. We will build a sequence of covers $\mathcal{U}_n$ for which we can apply Proposition \ref{Collnorm-sigmarefinementsTheorem} for each $n\ge 1$.

Let the cover $U_1$ be all the maximal elements of $\mathcal{B}$. Inductively we can build a cover $\mathcal{U}_n$ by first removing all the elements of $\mathcal{U}_k$ for $k < n$ from $\mathcal{B}$ and setting $\mathcal{U}_n$ to be all the maximal elements in the remianing collection.

Observe that $\mathcal{U}_n$ is point finite. Indeed, if $U$ and $V$ are in $\mathcal{U}_n$ and contain $x$ then by maximality $U \cap X\setminus V \neq \emptyset$ and $V \cap X\setminus U \neq \emptyset$. There are only finitely  many such elements for which that happens.

By Proposition \ref{Collnorm-sigmarefinementsTheorem} there is a $\mathcal{U}_n$-small partition of unity $\{f_s:X\to [0,1]\}_{s\in S_n}$. Observe $\bigcup\limits_{n=1}^{\infty}\mathcal{U}_n = \mathcal{B}$ and so $\{\frac{f_s}{2^n}:X\to [0,1]:s\in S_n$ and $n \in \mathbb{N}\}$ forms a partition of unity whose carriers form a basis for $X$. 

($\Rightarrow$) This direction follows from Arkhangelskii's Metrization Theorem.

\end{proof}

\end{document}